\documentclass{amsart}

\usepackage[colorlinks=true, pdfstartview=FitV, linkcolor=blue,citecolor=red, urlcolor=blue]{hyperref}
\usepackage[english]{babel}

\usepackage{amsmath,amsthm,amssymb}
\usepackage{caption}
\usepackage{subcaption}
\usepackage{color}
\usepackage{hyperref}
\usepackage{url}
\usepackage{graphicx, float}

\newcommand{\bburl}[1]{\textcolor{blue}{\url{#1}}}

\newtheorem{thm}{Theorem}[section]
\newtheorem{cor}[thm]{Corollary}
\newtheorem{lem}[thm]{Lemma}
\newtheorem{prop}[thm]{Proposition}

\theoremstyle{definition}

\theoremstyle{definition}
\newtheorem{defi}[thm]{Definition}
\theoremstyle{remark}
\newtheorem{rem}[thm]{Remark}

\newcommand\be{\begin{equation}}
\newcommand\ee{\end{equation}}
\numberwithin{equation}{section}

\newcommand{\R}{\ensuremath{{\bf R}}}

\newcommand{\N}{{\bf N}}

\begin{document}

\title[Biases in prime factorizations and Liouville functions]{Biases in prime factorizations and Liouville functions for arithmetic progressions}

\author{Peter Humphries}
\address{Department of Mathematics, University College London, Gower Street, London WC1E 6BT, United Kingdom}
\email{\textcolor{blue}{\href{pclhumphries@gmail.com}{pclhumphries@gmail.com}}}

\author{Snehal M. Shekatkar}
\address{Centre for Modeling and Simulation, S.P. Pune University, Pune, Maharashtra,
411007 India}
\email{\textcolor{blue}{\href{snehal.shekatkar@cms.unipune.ac.in}{snehal.shekatkar@cms.unipune.ac.in}}}

\author{Tian An Wong}
\address{Smith College, 44 College Lane, Northampton06013 MA, USA}
\email{\textcolor{blue}{\href{twong33@smith.edu}{twong33@smith.edu}}}

\subjclass[2010]{11A51, 11N13, 11N37, 11F66}

\keywords{Liouville function, prime factorization, arithmetic progressions, P\'{o}lya's conjecture}

\thanks{The authors thank the reviewer for helpful comments improving the exposition of the paper. The second author acknowledges the funding from the National Post Doctoral Fellowship (NPDF) of DST-SERB, India, File No. PDF/2016/002672. The third author would like to thank Neha Prabhu for helpful discussions concerning this work.}

\begin{abstract}
We introduce a refinement of the classical Liouville function to primes in arithmetic progressions. Using this, we show that the occurrence of primes in the prime factorizations of integers depend on the arithmetic progression to which the given primes belong. Supported by numerical tests, we are led to consider analogues of P\'{o}lya's conjecture, and prove results related to the sign changes of the associated summatory functions.
\end{abstract}

\maketitle

\bigskip

%\tableofcontents

%%%%%%%%%%%%%%%%%%%%%%%%%%%%%%%%%%%%%%%%%%%%%%%%%%%%%%%%%%%%%%%%%%%%%%%%%%%%%%%%%%%%%%%%%%%%%%%%%%%%%%%%%%%%%%%%%%%%%%
%%%%%%%%%%%%%%%%%%%%%%%%%%%%%%%%%%%%%%%%%%%%%%%%%%%%%%%%%%%

\section{Introduction}

\subsection{The Liouville function}

The classical Liouville function is the completely multiplicative function defined by $\lambda(p)=-1$ for any prime $p$. It can be expressed as $\lambda(n)=(-1)^{\Omega(n)}$ where $\Omega(n)$ is the total number of prime factors of $n$. One sees that it is $-1$ if $n$ has an odd number of prime factors, and 1 otherwise. By its relation to the Riemann zeta function
\be
\label{liouvillezeta}
\sum_{n=1}^\infty \frac{\lambda(n)}{n^s}=\frac{\zeta(2s)}{\zeta(s)},
\ee
the Riemann hypothesis is known to be equivalent to the statement that 
\be
\label{liouvilleRH}
L(x) := \sum_{n\leq x} \lambda(n)= O_{\epsilon}(x^{1/2+\epsilon})
\ee
for any $\epsilon>0$ (see, for example, \cite[Theorem 1.1]{Hu}); whereas the prime number theorem is equivalent to the estimate $o(x).$ Indeed, the behaviour of the Liouville function, being a close relative of the more well-known M\"obius function, is strongly connected to prime number theory. Also, we note that by the generalized Riemann hypothesis, one also expects \eqref{liouvilleRH} to hold for partial sums of $\lambda(n)$ restricted to arithmetic progressions, with an added dependence on the modulus.

In this paper, we introduce natural refinements of the Liouville function, which detect how primes in given arithmetic progressions appear in prime factorizations. We find that that these functions behave in somewhat unexpected ways, which is in turn related to certain subtleties of the original Liouville function.

Define $\Omega(n;q,a)$ to be the total number of prime factors of $n$ congruent to $a$ modulo $q$, and
\be
\lambda(n;q,a) = (-1)^{\Omega(n;q,a)}
\ee
to be the completely multiplicative function that is $-1$ if $n$ has an odd number of prime factors congruent to $a$ modulo $q$, and $1$ otherwise. They are related to the classical functions by
\be
\label{lambdaOmega}
\lambda(n)=\prod_{a=0}^{q-1}\lambda(n;q,a), \quad \Omega(n)=\sum_{a=0}^{q-1}\Omega(n;q,a).
\ee
Using this we study the asymptotic behaviour instead of
\be
L(x;q,a)=\sum_{n\leq x}\lambda(n;q,a),
\ee
hence the distribution of the values of $\lambda(n;q,a)$. Also, we will be interested in $r$-fold products of $\lambda(n;q,a)$,
\be
\lambda(n;q,a_1,\dots,a_r) = \prod_{i=1}^r\lambda(n;q,a_i)
\ee
where the $a_i$ are distinct residue classes modulo $q$, with $1\leq r \leq q$, and define $\Omega(n;q,a_1,\dots,a_r)$ and $L(x;q,a_1,\dots,a_r)$ analogously.

\subsection{Prime factorizations}

Given a prime number $p$, we will call the \emph{parity} of $p$ in an integer $n$ to be even or odd according to the exponent of $p$ in the prime factorization of $n$. This includes the case where $p$ is prime to $n$, in which case its exponent is zero and therefore having even parity.

Landau \cite{Landau} proved that the number of $n\leq x$ containing an even (resp. odd) number of prime factors both tend to
\be
\label{landau}
\frac12 x + O(xe^{-c\sqrt{\log x}})
\ee
with $x$ tending to infinity, and $c$ some positive constant. In fact, he showed that this is equivalent to the prime number theorem. P\'olya \cite{Po} asked whether $L(x)$ is nonpositive for all $x\ge 2$. A negative answer to this question was provided by Haselgrove \cite{Ha}, building on the work of Ingham \cite{I}, using the zeroes of $\zeta(s)$; in fact, the sum must change sign infinitely often, with the first sign change subsequently computed to be around $9\times 10^8$. A similar problem was posed by Tur\'an on the positivity of partial sums of $\lambda(n)/n$, which was also shown to be false, with the first sign change taking place around $7\times 10^{13}$ \cite[Theorem 1]{BFM}.

On the other hand, by the equidistribution of primes in arithmetic progressions, one might guess that the number of $n\leq x$ containing an even (respectively odd) number of prime factors $p \equiv a \pmod{q}$, for a fixed arithmetic progression would be evenly distributed over residue classes coprime to $q$. By our analysis of $\lambda(n;q,a)$, we find this not to be the case. Our first theorem shows a quantitative difference in taking all residue or non-residue classes.

\begin{thm}
\label{nonres}
Given any $q\geq 2$, let $a_1,\dots,a_{\varphi(q)}$ be the residue classes modulo $q$ such that $(a_i,q)=1$, and $b_1,\dots,b_{q-\varphi(q)}$ the remaining residues classes. Then 
\be
\label{residueclasses}
\sum_{n\leq x}\lambda (n;q,a_1,\dots,a_{\varphi(q)}) = o(x),
\ee
for $x\geq 1$. Assuming the Riemann hypothesis, \eqref{residueclasses} is $O_{\epsilon}(x^{1/2+\epsilon})$ for all $\epsilon>0$. On the other hand,
\be
\sum_{n\leq x}\lambda(n;q,b_1,\dots,b_k) = \left(\prod_{i=1}^k\prod_{p \mid b_i}\frac{p-1}{p+1}\right)x+o(x)
\ee
for any subset of residue classes $b_1,\dots,b_{q-\varphi(q)}$ of size $k$.
\end{thm}

\noindent In fact, it is straightforward to show that the estimate $O_{\epsilon}(x^{1/2+\epsilon})$ is equivalent to the Riemann hypothesis, as in the classical Liouville function.

\begin{figure}
\begin{center}
\includegraphics[width=\columnwidth]{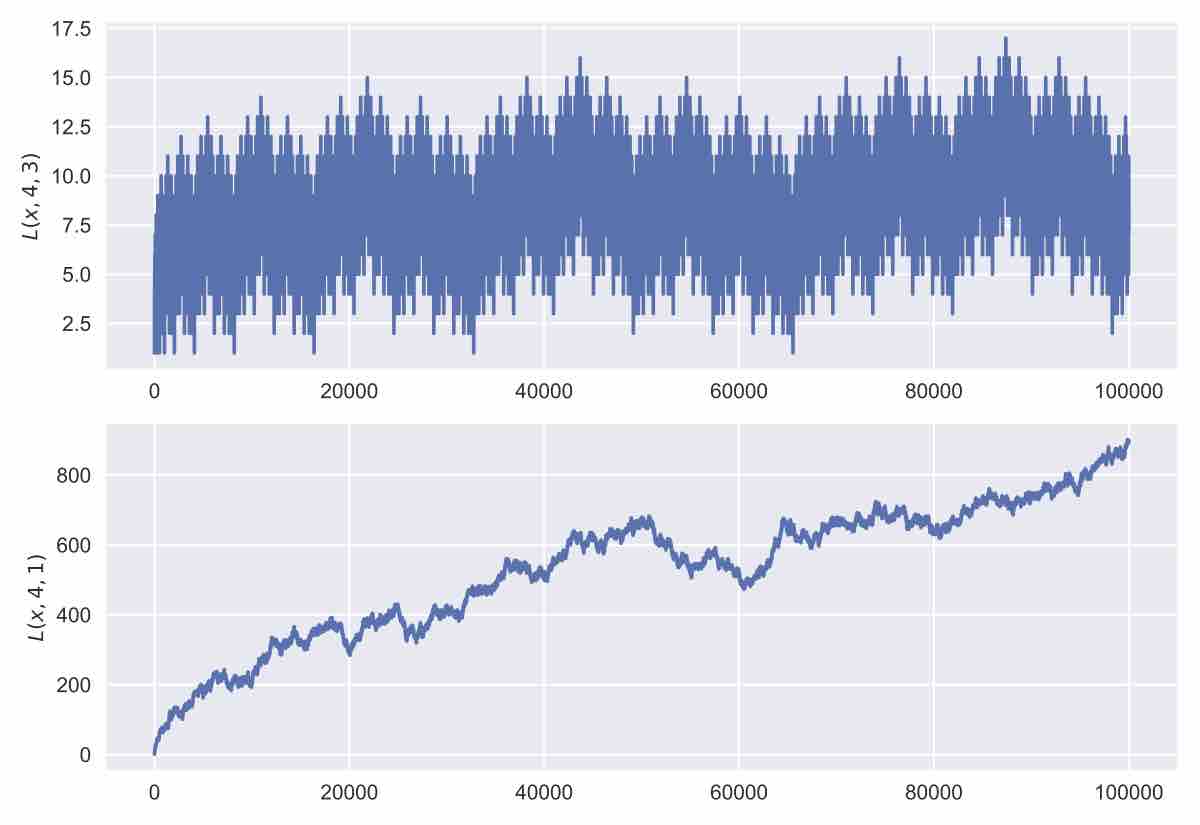}
\caption{$L(x;4,3)$ and $L(x;4,1)$.}
\end{center}
\end{figure}

\begin{defi}
\label{characterlikedef}
For general moduli, taking $r= \varphi(q)/2$, there exists a choice of $a_1,\dots,a_r$ coprime to $q$, we have that $\lambda(n;q,a_1,\dots,a_r) = \chi_q(n)$ for some non-trivial real Dirichlet character mod $q$, whenever $(n,q) = 1$. We will call these Liouville functions \emph{character-like}\footnote{See also \cite[p.2]{BCC}.}, and its \emph{complement} to be $\lambda(n;q,b_1,\dots,b_{q-r})$. 
\end{defi}
%PH: I don't understand this statement, and I couldn't find anything in \cite{BCC} regarding `character-like' functions. Maybe rewrite this? Alternatively, give a nice definition of character-like. Actually, I think this should be explained in some detail at some point.
%TA fixed Actually, this was in the abstract of BCC!

By multiplicativity, our function $\lambda(n;q,a_1,\dots,a_r)$ is defined by its values on primes less than $q$, so it suffices to set these values to be equal to that of the desired non-trivial real Dirichlet character $\chi_q$. In this case, we can predict the behaviour of the function and its `complement', so to speak. For example, the $\lambda(n;4,3)$ resembles the nonprincipal Dirichlet character modulo $4$, and its partial sums are shown to be \emph{positive}, whereas the $\lambda(n;4,1)$ turns out to be related to the classical $\lambda(n)$ restricted to arithmetic progressions modulo $4$ (c.f. Proposition \ref{primesmod4}). The next theorem describes the character-like case and its complement.

\begin{thm}
\label{varphi2}
Fix $q\ge2$, and let $r=\varphi(q)/2$. Also let $a_1,\dots,a_{q-r}$ and $b_1,\dots,b_r$ be chosen as in Definition \ref{characterlikedef}. Then
\be
\label{resgeneral}
\sum_{n\leq x}\lambda(n;q,a_1,\dots,a_{q-r}) =O_{\epsilon}(x^{1/2+\epsilon})
\ee
is equivalent to the generalised Riemann hypothesis for $L(s,\chi_q)$; unconditionally, it is $o(x)$. On the other hand,
\be
\label{dirichletgeneral}
\sum_{n\leq x}\lambda(n;q,b_1,\dots,b_r) = O(\log x).
\ee
if there is only one nonprincipal real Dirichlet character modulo $q$; if there is more than one, then we have only $o(x^{1-\delta})$ for some $\delta>0$.
\end{thm}

\begin{figure}
\begin{center}
\includegraphics[width=\columnwidth]{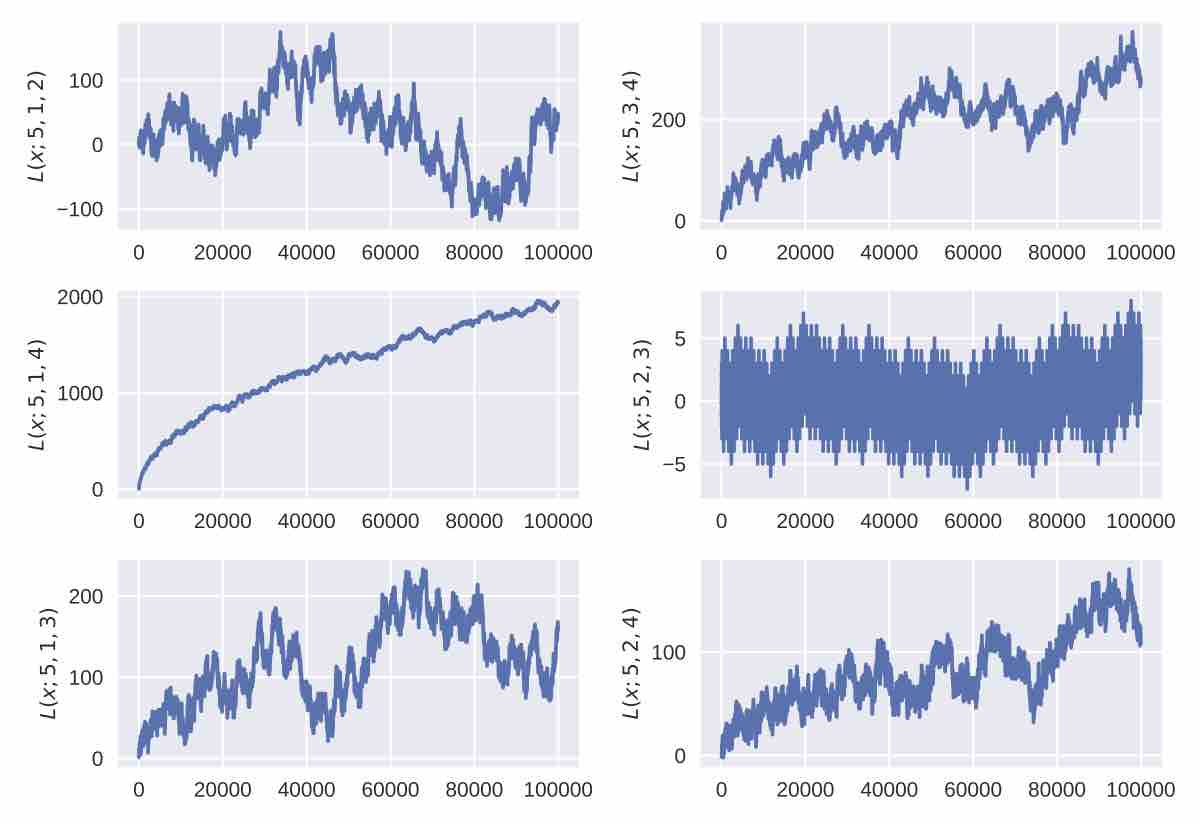}
\caption{Combinations of $L(x;5,*,*)$.}
\end{center}
\end{figure}

Moreover, we show in the next theorem that when $r\neq \varphi(q)/2$, and $(a_i,q)=1$, the behaviour of $\lambda(n;q,a_1,\dots,a_r)$ is determined. Otherwise, the behaviour of $\lambda(n;q,a_1,\dots,a_r)$ seems more difficult to describe precisely, and in this case it is interesting to ask the same question as P\'olya did for $\lambda(n)$. For example, with modulus 5 we observe as in Figure 2, that except for the character-like function and its complement, the partial sums tend to fluctuate with a positive bias, except for $\lambda(n;5,1,2),$ which already changes sign for small $x$. The remaining three remain positive up to $x\leq 10^7$, which leads us to ask whether they eventually change sign.

\begin{thm}
\label{karatsuba}
Let $a_1,\dots,a_r$ be distinct residue classes modulo $q$, coprime to $q$. Then for $r\neq \varphi(q)/2,\varphi(q)$,
\be
\sum_{n\leq x}\lambda(n;q,a_1,\dots,a_r)
= b_0 \frac{x}{(\log x)^{2 - \frac{2r}{\varphi(q)}}} +O\left(\frac{x}{(\log x)^{3 - \frac{2r}{\varphi(q)}}}\right),
\ee
where $b_0$ is an explicit constant such that $b_0>0$ if $2r<\varphi(q)$ and $b_0<0$ if $2r>\varphi(q)$. If $r=\varphi(q)/2$ and $\lambda(n;q,a_1,\dots,a_r)$ is not character-like, we have again
\be
\sum_{n\leq x}\lambda (n;q,a_1,\dots,a_{\varphi(q)}) = o(x),
\ee
for $x\geq 1$. Assuming the Riemann hypothesis, it is in fact $O_{\epsilon}(x^{1/2+\epsilon})$ for all $\epsilon>0$.
\end{thm}

\noindent The proof of the first estimate essentially follows the method Karatsuba for the Liouville function \cite{K}, which is a simple variant of the Selberg--Delange method.

The most intriguing aspect of our new family of Liouville-type functions, in light of the conjectures of P\'olya, Tur\'an, and even Mertens, is distinguishing when the partial sums of $\lambda(n;q,a_1,\dots,a_r)$ have any sign changes at all and furthermore whether infinitely many sign changes must follow. We give conditional answers to this question in Section \ref{biases} in the particular case where $r = \varphi(q)/2$ and $\lambda(n;q,a_1,\dots,a_r)$ is the complement to a character-like function. Our main result concerns the logarithmic density
\be
\delta(P) := \lim_{X \to \infty} \frac{1}{\log X} \int\limits_{P \cap [1,X]} \, \frac{dx}{x},
\ee
of the set
\be
P := \left\{x \in [1,\infty) : \sum_{n\leq x}\lambda(n;q,a_1,\dots,a_{q-r}) \geq 0\right\}.
\ee
To study this, we must assume the following conjecture.

\begin{defi}
We say that $L(s,\chi_q)$ satisfies the linear independence hypothesis if the set
\be
\left\{\gamma \geq 0 : L\left(\frac{1}{2} + i\gamma,\chi_q\right) = 0\right\}
\ee
is linearly independent over the rationals.
\end{defi}

In particular, the linear independence hypothesis implies that $L(1/2,\chi_q) \neq 0$ and that every zero is simple.

\begin{thm}
\label{mainbiasesthm}
Assume the generalised Riemann hypothesis and linear independence hypothesis for $L(s,\chi_q)$, and that the bound
\be
\sum_{0 < \gamma \leq T} \frac{1}{|L'(\rho,\chi_q)|^2} \ll T^{\theta}
\ee
holds for some $1 \leq \theta < 3 - \sqrt{3}$. Then
\be
\frac{1}{2} \leq \delta(P) < 1.
\ee
\end{thm}

That is, the limiting logarithmic density of $P$ is at least $1/2$ but strictly less than $1$, so that `most' of the time, $\sum_{n\leq x}\lambda(n;q,a_1,\dots,a_{q-r})$ is nonnegative, but nevertheless it is negative a positive proportion of the time.

\begin{rem}
The requirement that $\theta < 3 - \sqrt{3}$ may be weakened to $\theta < 2$; cf.~\cite[Remark 3]{Meng1}. Conjecturally, one expects the exponent $\theta = 1$ to be both valid and sharp; cf.~\cite[Conjecture 1.3]{HKO}.
\end{rem}

\subsubsection{Auxillary results}

We also prove some results on distribution of total number of primes in arithmetic progression; the analogues $\Omega(n;q,a)$ and $\omega(n;q,a)$ are more well-behaved, though we still observe some slight discrepancy in the implied constants in the growth of the partial sums, with respect to the residue class. For example, we show that
\be
\sum_{n\leq x} \omega(n;q,a) =\frac{1}{\varphi(q)}x\log \log x + g(q,a) x + o(x)
\ee
for some absolute constant $g(q,a)$ (see Proposition \ref{omegaloglog}), and that $\omega(n;q,a)$ is distributed normally, as an application of the Erd\H{o}s-Kac theorem.

\subsubsection{Mixed residue classes}

Finally, we mention the `mixed' case, where $\lambda(n;a_1,\dots,a_r)$ involves both residue classes that are and are not coprime to $q$. Numerical experiments seem to suggest that they do affect the behaviour in small but observable ways; in particular, we observe that adding several residue classes may cause a sum to fluctuate more. For example, in the Figure 3 above, the addition of residue classes 2 and 3, which divide 6, affect the fluctuations in the sum in a nontrivial manner. In fact, while we know that $L(x;6,1)$ is asymptotically positive, $L(x;6,1,2,3)$ already exhibits multiple sign changes.

\begin{figure}
\begin{center}
\includegraphics[width=\columnwidth]{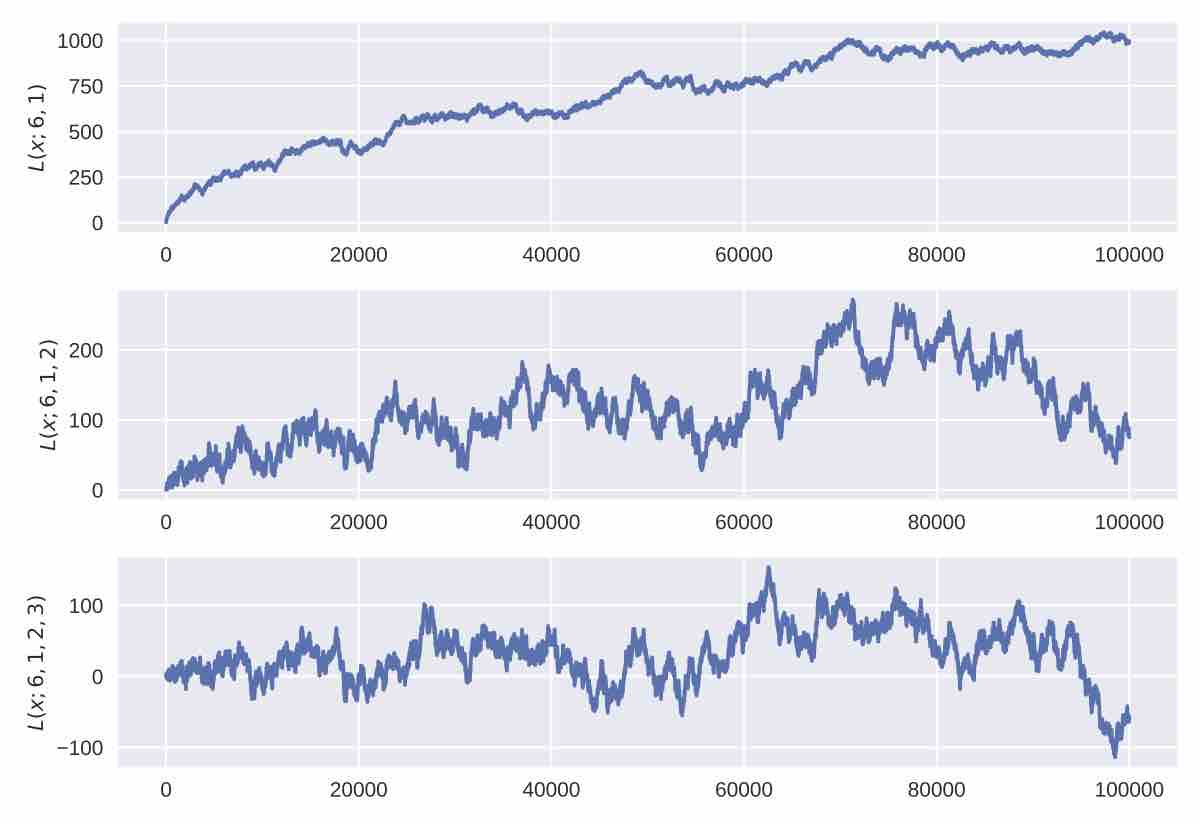}
\caption{Mixing residue classes.}
\end{center}
\end{figure}

\subsubsection{Outline} This paper is organized as follows. In Section \ref{first} we develop basic properties of the Liouville function for arithmetic progressions, including brief discussions of the analogue of the Chowla conjecture, and the distribution of the primes in arithmetic progressions in a given prime factorization. 

In Section \ref{parity}, we study the parity of the number of primes in arithmetic progressions appearing in a given prime factorization, and prove Theorems \ref{nonres}, \ref{varphi2} and \ref{karatsuba}. 

In Section \ref{biases}, we study the occurrence of sign changes and prove Theorem \ref{mainbiasesthm}.

%\emph{Notation.} Throughout the paper we fix the following conventions: $\epsilon$ will denote a positive real number and $c$ an absolute constant, which will vary depending on the context.

%%%%%%%%%%%%%%%%%%%%%%%%%%%%%%%%%%%%%%%%%%%%%%%%%%%%%%%%%%%%
% FIRST ESTIMATES
%%%%%%%%%%%%%%%%%%%%%%%%%%%%%%%%%%%%%%%%%%%%%%%%%%%%%%%%%%%%

\section{First estimates}
\label{first}

\subsection{Basic properties}

We develop some basic properties of the Liouville function for arithmetic progressions, analogous to the classical results. Using this, we prove a basic estimate for the distribution of $\lambda(n;q,a)$. Most of the statements in this section will be proven for $\lambda(n;q,a)$, and we leave to the reader the analogous statements for products $\lambda(n;q,a_1,\dots,a_r)$. 

In \cite{BCC}, the authors consider any subset $A$ of prime numbers, and define $\Omega_A(n)$ to be the number of prime factors of $n$ contained in $n$, counted with multiplicity. They then define a Liouville function for $A$ to be $\lambda_A(n)=(-1)^{\Omega_A(n)},$ taking values $-1$ at primes in $A$ and 1 at primes not in $A$, and show, for example, that $\lambda_A(n)$ is not eventually periodic in $n$. Our functions can be viewed as particular cases of $\lambda_A(n)$ where $A$ is a set of primes in a given arithmetic progression. This by \cite[Theorem 1]{BCC}, we infer that  $\lambda(n;q,a)$ is not eventually periodic.

Secondly we note that its partial sums are unbounded from the Erd\H{o}s discrepancy problem,\footnote{This may seem like big hammer to invoke, but it is worth noting that our $\lambda(n;q,a)$ have as special cases the `character-like' multiplicative functions considered in \cite{BCC}, whose $O(\log x)$ growth constitute near misses to the problem.} but we will also prove this more directly below. 

\subsubsection{Divisor sum}

Recall that the classical Liouville function satisfies the identity
\be
\sum_{d|n}\lambda(d)=
\begin{cases}
1& \text{if $n$ is a perfect square,}\\
0& \text{otherwise.}
\end{cases}
\ee
The following proposition gives the analogue of this identity.

\begin{lem}
Write $n=n_1n_2$ with $n_1$ not divisible by any prime $p\equiv a\pmod{q}$ dividing $n$. Then we have
\be
\label{sum}
S(n;q,a) := \sum_{d|n}\lambda(d;q,a)=
\begin{cases}
\tau(n_1) & \text{if $n_2$ is a perfect square,}\\
0& \text{otherwise.}
\end{cases}
\ee
where $\tau(n)$ is the divisor function.
\end{lem}

\begin{proof}
By multiplicativity, it suffices to prove this when $n = p^r$ is a prime power. If $p \equiv a \pmod{q}$, then
\be
S(p^r;q,a) = \sum_{i = 0}^{r} \lambda(p^i;q,a) = \sum_{i = 0}^{r} (-1)^i,
\ee
which is $1$ is $r$ is even, so that $p^r$ is a perfect square, and is $0$ otherwise. If $p \not\equiv a \pmod{q}$, on the other hand, then it is trivially true that $S(p^r;q,a) = r + 1 = \tau(p^r)$.
%%%We will proceed by induction on $k$, writing $n_2=p_1^{r_1}\dots p_k^{r_k}$. First suppose $k=0$, then by definition $\lambda(n_1;q,a)=1$, and therefore $S(n_1;q,a)=\tau(n_1)$. On the other hand, if $k=1$ and $r_1$ is odd, then $n=n_1p_1^{r_1}$ and
%%%\be
%%%S(n;q,a)=\sum_{d|a} \sum_{i=0}^{r_1} \lambda(n_1p^i_1;q,a)=\sum_{d|a} \lambda(n_1)\sum_{i=0}^{r_1} (-1)^i,
%%%\ee
%%%and the inner sum vanishes since $r_1$ is odd.
%%%
%%%Now suppose that \eqref{sum} holds for $k-1$, and define $A$ such that $n=Ap^{r_k}_k$. Then
%%%\be
%%%S(n;q,a)=\sum_{d|A}\sum_{i=0}^{r_k}\lambda(Ap^{i}_k;q,a),
%%%\ee
%%%which, by multiplicativity, is
%%%\be
%%%\label{sumproof}
%%%\Big(\sum_{d|A}\lambda(A;q,a)\Big)\Big(\sum_{i=0}^{r_k}\lambda(p^{i}_k;q,a)\Big)=S(A;q,a)\sum_{i=0}^{r_k}(-1)^i,
%%%\ee
%%%since $\lambda(p^{i}_k;q,a)=(-1)^i$. Now we see that if $n_2$ is a perfect square, the sum is equal to 1 and $S(A;q,a)=\tau(A)$ by hypothesis; whereas if $n_2$ is not a perfect square, either $r_k$ is odd and the sum vanishes, or $r_k$ is even and $r_{k'}$ is odd for some $k'<k$, in which case $S(A;q,a)=0$ by hypothesis.
%%PH:%% This is much longer than needed.
\end{proof}

%From this we have a crude first estimate:
%
%\begin{cor}
%$L(x;q,a) = O(x\log x)$ for any $a,q$, and $x>1$.
%\end{cor}
%
%\begin{proof}This follows from the M\"obius inversion formula
%\be
%\lambda(n;q,a)=\sum_{d|n}\mu\left(\frac{n}{d}\right)S(d;q,a),
%\ee
%and by applying the trivial bounds on $\mu(n)$ and $S(n;q,a)$.
%\end{proof} 
%%PH:%% This is worse than the trivial bound!!! In any case, is this the only reason for including $S(n;q,a)$? I don't really see how it has any relevance otherwise.

%%%%%%%%%%%%%%%%%%%%%%%%%%%%%%%%%%%%%%%%%%%%%%%%%%%%%%%%%%%%%%%%%%%%%%%%%%%%%%%%%%%%%%%%%%%%%%%%%%%%%%%%%%%%%%%%%%%%%%%%%%
%Notice that since $S(n;q,a)$ is nonnegative, and in particular $\geq1$ when $n$ is a perfect square, we have
%\be
%\sum_{n\leq x}1*\lambda (n)  \geq \sum_{n\leq \sqrt{x}} 1 = \sqrt{x} + O(1).
%\ee
%Then by the Dirichlet hyperbola method,
%\be
%\sum_{n\leq x} 1*\lambda =  \sum_{a \leq \sqrt x} \lambda(a) \sum_{b \leq x/a} 1 + \sum_{b\leq \sqrt x} 1 \sum_{\sqrt x < a \leq x/b } \lambda(a)
%\ee
%%%%%%%%%%%%%%%%%%%%%%%%%%%%%%%%%%%%%%%%%%%%%%%%%%%%%%%%%%%%%%%%%%%%%%%%%%%%%%%%%%%%%%%%%%%%%%%%%%%%%%%%%%%%%%%%%%%%%%%%%%%%%%%%

\subsubsection{Average orders}

Call a subset $A$ of primes to have sifting density $\kappa$, if
\be
\sum_{\substack{p\leq x \\ p \in A}}\frac{\log p}{p} = \kappa \log x + O(1).
\ee
where $0\leq \kappa\leq 1$. In particular, we will take $A$ to be the set of primes congruent to a certain $a$ modulo $q$. For example, we have $\kappa=0$ if $(a,q)>1$, and for $(a,q)=1$ with $q$ odd, we have $0<\kappa\leq\frac12$ with equality only when $\varphi(q)=2$. 

\begin{prop}
We have for any $a,q$ such that $\kappa <\frac12$ and $(a,q)=1$,
\be
\sum_{n\leq x}\lambda(n;q,a) = (1+o(1)) \frac{C_{\kappa} x}{(\log x)^{2\kappa}}
\ee
where $C_\kappa> 0$ is an explicit constant depending on $a,q$ and $\kappa$, and for $\kappa \geq \frac12$,
\be
\sum_{n\leq x}\lambda(n;q,a)=o(x).
\ee
\end{prop}

\begin{proof}
This follows from \cite[Theorem 5]{BCC}, as an application of the Liouville function for $A$, choosing $A$ to be a set of primes in arithmetic progression. 
\end{proof}

\begin{rem}
More generally, if we take $A$ to consist of several residue classes $a_i$ modulo $q$, then $\kappa$ will also vary accordingly according to the number of residue classes prime to $q$ that are taken. In this case, we may replace the sum over $\lambda(n;q,a)$ by $\lambda(n;q,a_1,\dots,a_r)$ to obtain similar estimates.
\end{rem}

We also record the following estimates for the classical Liouville function as a benchmark: The average order 
$\sum_{n\leq x}\lambda(n) = O(x e^{-c_1\sqrt{\log x}})$ is well known, and by the same method of proof of \cite[Theorem 2]{C}, we have that
\be
\sum_{\substack{n\leq x\\ n \in P}} \lambda(n)  = -c_2 \frac{x}{(\log x)^{\frac{r}{\varphi(q)}+1}} + O\left(\frac{x}{(\log x)^{\frac{r}{\varphi(q)}+2}}\right),
\ee
where $P$ is a set of $r$ residue classes coprime to $q>2$ and $c_1,c_2>0$.

In particular, we observe that when $\lambda(n)$ is restricted to arithmetic progressions (containing infinitely many primes), its partial sums tend to be negative. One can also show that its limiting distribution is negative using its relation to Lambert series; see \cite[Theorem 1]{BL}.

\subsection{Chowla-type estimates}

One may also consider a variant of the Chowla conjecture for the Liouville function. Fix $a,q$ relatively prime. Given distinct integers $h_1,\dots,h_k$, fix a sequence of signs $\epsilon_j=\pm1$ for $1\leq j \leq k$. Then one would like to know whether
\be
\sum_{n\leq x}\prod_{i=1}^k\lambda(n+h_i;q,a) = o(x).
\ee
for all $k$. In particular, the number of $n\leq x$ such that $\lambda(n+j;q,a)=\epsilon_j$ for all$ 1\leq j\leq k$ is
$(1/2^{k}+o(1))x.$ Roughly, this tells us that $\lambda(n;q,a)$ takes the value 1 or $-1$ randomly. We then have the following evidence towards the conjecture.

\begin{prop}
\label{chowla}
For every $h\geq 1$ there exists $\delta(h)>0$ such that
\be
\frac{1}{x}\left| \sum_{n\leq x}\lambda(n;q,a)\lambda(n+1;q,a)\right| \leq 1 - \delta(h)
\ee
for all sufficiently large $x$. Similarly,
\be
\frac{1}{x}\left| \sum_{n\leq x}\lambda(n;q,a)\lambda(n+1;q,a)\lambda(n+2;q,a)\right| \leq 1 - \delta(h)
\ee
\end{prop}

\begin{proof}
This follows as an application of \cite[Theorem 1]{MR} to deduce the analog of \cite[Corollary 2]{MR}, in our case specializing the multiplicative function $f(n)$ to be $\lambda(n;q,a)$ instead of $\lambda(n)$.
\end{proof}

\begin{rem}
We also note that as an application of \cite[Corollary 5]{MR} on sign changes of certain multiplicative functions, there exists a constant $C$ such that every interval $[x,x+C\sqrt{x}]$ contains a number with an even number of prime factors in a fixed arithmetic progression $a$ modulo $q$, and another one with an odd number of such prime factors.
\end{rem}

\subsection{The number of prime factors}

Recall the functions $\omega(n)$ counting the number of distinct prime factors of $n$, and $\Omega(n)$ counting the total number of prime factors of $n$. 
%We may express them as
%\be
%\omega(n) = \sum_{p \mid n}1, \quad \Omega(n) = \sum_{j=1}^\infty \sum_{p : p^j | n} 1.
%\ee
%%PH:%% This isn't relevant (and in any case, the latter is horribly defined).
As in \eqref{lambdaOmega}, we may define the analogous functions $\omega(n;q,a)$ and $\Omega(n;q,a)$ counting only primes congruent to $a$ modulo $q$, so that
\be
\omega(n)=\sum_{a=0}^{q-1}\omega(n;q,a), \quad \Omega(n)=\sum_{a=0}^{q-1}\Omega(n;q,a).
\ee

\begin{prop}
\label{omegaloglog}
There exists an absolute constant $g(q,a)$ such that
\be
\sum_{n\leq x} \omega(n;q,a) =\frac{1}{\varphi(q)}x\log \log x + g(q,a) x + o(x).
\ee
%%PH:%% Fixed this, as previously was incorrect.
\end{prop}

\begin{proof}
Write
\be
\sum_{n\leq x} \omega(n;q,a) = \sum_{n\leq x}\sum_{\substack{p \mid n\\ p\equiv a \hspace{-.2cm} \pmod{q} }}1 = \sum_{\substack{p\leq x \\ p\equiv a \hspace{-.2cm} \pmod{q} }}\sum_{m \leq x/p}1
\ee
which is
\be
\sum_{\substack{p\leq x \\ p\equiv a \hspace{-.2cm} \pmod{q} }} \frac{x}{p} + O\left(\sum_{\substack{p\leq x \\ p\equiv a \hspace{-.2cm} \pmod{q} }}1\right) = \sum_{\substack{p\leq x \\ p\equiv a \hspace{-.2cm} \pmod{q} }} \frac{x}{p} + O\left(\frac{x}{\log x}\right)
\ee
by the prime number theorem. Now, using Mertens' theorem for primes in arithmetic progressions for $(a,q)=1$ \cite[Corollary 4.12]{MV}, we have
\be
\sum_{\substack{p \leq x \\ p \equiv a \hspace{-.2cm} \pmod{q}}} \frac{1}{p} = \frac{1}{\varphi(q)}\log\log x + g(q,a) + o(1)
\ee
where $g(q,a)$ is an absolute constant. Applying this yields the proposition.
%g(q,a) = \frac{\gamma}{\varphi(q)} +\frac{1}{\varphi(q)}\sum_{\chi\neq \chi_0}\chi(\bar{a}) \log (\frac{\varphi(q)}{q} L(1,\chi)) - \sum_{p^m\equiv a \hspace{-.2cm} \pmod{q}}\sum_{m=2}^{\infty} \frac{1}{mp^m}
\end{proof}

\begin{figure}
\begin{center}
\includegraphics[width=\columnwidth]{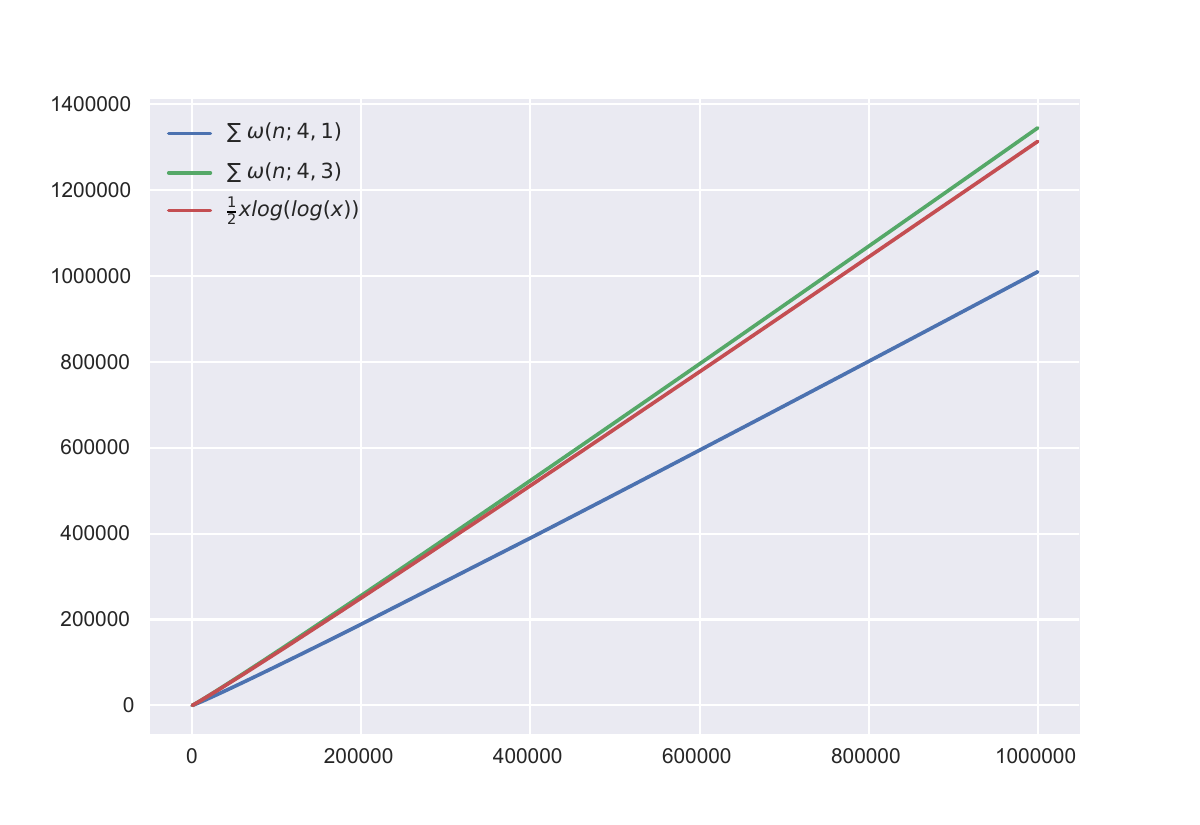}
\caption{Distribution of $\omega(n;q,a)$.}
\end{center}
\end{figure}

Note that this constant has been studied in greater detail, for example \cite{LZ1, LZ2, Meng2}, which may be useful in further analysis of the partial sums of $\omega(n;q,a)$.\footnote{We thank the referee for pointing this out to us.} One may also consider higher moments, such as
\be
\sum_{n\leq x}(\omega(n;q,a) - \frac{1}{\varphi(q)}\log\log x)^2=\frac{1}{\varphi(q)}x\log\log x + O(x)
\ee
by expanding the square and applying simple estimates. 

Moreover, since $\omega(n;q,a)$ is additive, we may apply the Erd\H{o}s--Kac theorem \cite{EK}, which applies to strongly additive functions---additive functions $f$ such that $f(mn)=f(m)+f(n)$ for all natural numbers $m,n$, and $|f(p)|\leq 1$ for all primes $p$---to immediately obtain the following statement.

\begin{thm}
[Erd\H{o}s--Kac]
Fix a modulus $q$ and constants $A,B\in \R$. Then
\begin{align}
&\lim_{x\to \infty }\frac{1}{x} \#\left\{n\leq x : A \leq \frac{\omega(n;q,a)-\frac{1}{\varphi(q)}\log\log x}{\sqrt{\frac{1}{\varphi(q)}\log\log x}}\leq B \right\} \\
&= \frac{1}{\sqrt{2\pi}}\int^B_A e^{-t^2/2}dt.
\end{align}
\end{thm}

\noindent Hence $\omega(n;q,a)$ is also normally distributed.

\section{Parity of prime factors}
\label{parity}

We now turn to the average behaviour of our $\lambda(n;q,a)$. A refinement of P\'olya's problem leads us to ask: given the prime factorization of a composite number $n$, do the primes in arithmetic progressions tend to appear an even or odd number of times? As described in the introduction, we show that one encounters biases, namely, that the answer depends strongly on the arithmetic progression chosen.

\subsection{Dirichlet series}

Since $\lambda(n;q,a)$ is completely multiplicative, we can form the Dirichlet series generating function
\be
\label{dsqa}
D(s;q,a) := \sum_{n=1}^\infty \frac{\lambda(n;q,a)}{n^s}=\zeta(s)\prod_{p\equiv a \hspace{-.2cm} \pmod{q}}\frac{1-p^{-s}}{1+p^{-s}}
\ee
using the Euler product in the case $(a,q)=1$, and by the trivial bound converges absolutely for $\Re(s)>1$. Taking the product over all such $a$, we obtain
\be
\label{phiproduct}
\prod_{\substack{a \hspace{-.2cm} \pmod{q} \\ (a,q)=1}} D(s;q,a) = \zeta(2s) \zeta(s)^{\varphi(q)-2} \prod_{p \mid q}\frac{1+p^{-s}}{1-p^{-s}}.
\ee
and we see that in the region $\Re(s) > 0$, the expression has a pole of order $\varphi(q)-2$ at $s=1$ and a simple pole at $s=\frac12$. Moreover, if we include residue classes $a$ modulo $q$ such that $(a,q)>1$, for which $D(s;q,a)$ is equal to $\zeta(s)$ up a finite number of factors, we have
\be
\label{allproduct}
\prod_{a=0}^{q-1} D(s;q,a) = \zeta(2s)\zeta(s)^{q-2},
\ee
generalizing the classical formula with $q = 1$.

Similarly, for products $\lambda(n;q,a_1,\dots,a_r)$ with $(a_i,q)=1$ for each $i$, we have
\be
D(s;q,a_1,\dots,a_r) = \zeta(s)\prod_{i=1}^r\prod_{p\equiv a_i \hspace{-.2cm} \pmod{q}}\frac{1-p^{-s}}{1+p^{-s}}
\ee
and
\be
D(s;q,a_1,\dots,a_r)D(s;q,a_1',\dots,a'_{\varphi(q)-r})  = \zeta(2s)\prod_{p \mid q}\frac{1+p^{-s}}{1-p^{-s}}.
\ee
where $a_1',\dots,a_{\varphi(q)-r}'$ are the remaining residue classes coprime to $q$.

It is known that for $(a,q)=1$, the Euler product
\be
F_a(s)=\prod_{p\equiv a \hspace{-.2cm} \pmod{q}}\frac{1}{1-p^{-s}}
\ee
converges absolutely for $\Re(s)>1$, and has analytic continuation to $\Re(s)\geq 1-C/\log t$ for $|t|<T,$ and $T\geq 10$ \cite[p.212]{K}. It can be expressed as 
\be
\label{faga}
\prod_{p\equiv a \hspace{-.2cm} \pmod{q}}\frac{1}{1-p^{-s}}=\zeta(s)^{\frac{1}{\varphi(q)}} e^{G_a(s)},
\ee
where $G_a(s)$ is given by
\begin{align}
\label{G_a}
&\frac{1}{\varphi(q)}\sum_{\substack{\chi \hspace{-.2cm} \pmod{q} \\ \chi\neq\chi_0}} \overline{\chi}(a)\left(\log L(s,\chi)+\sum_p\sum_{m=2}^{\infty}\frac{\chi(p)-\chi^m(p)}{mp^{ms}}\right) \\
&+ \frac{1}{\varphi(q)}\sum_{p \mid q}\log (1-p^{-s}).
\end{align}

Even though $\lambda(n;q,a)$ is not eventually periodic, it is still related to Dirichlet characters by the following identity. 

\begin{prop}
\label{characterlike}
Let $q>2$. Given any nonprincipal real Dirichlet character $\chi_q$ modulo $q$, there is a combination of residue classes $a_1,\dots, a_r$ such that
\be
D(s;q,a_1,\dots,a_r) = L(s,\chi_q) \prod_{p \mid q}\frac{1}{1-p^{-s}}.
\ee
\end{prop}

\begin{proof}
Given $\chi_q$, we can choose a combination of residue classes $a_1,\dots,a_r$, coprime to $q$ for which $\chi(a_i)=-1$, so that $\lambda(n;q,a_1,\dots,a_r) = \chi_q(n)$ whenever $(n,q) = 1$. Then we may express the Dirichlet series as
\begin{align}
D(s;q,a_1,\dots,a_r) 
&= \prod_{p} \frac{1}{1-\lambda(p;q,a_1,\dots,a_r)p^{-s}}\\
&= \prod_{p \nmid q} \frac{1}{1-\chi_q(p)p^{-s}} \prod_{p \mid q} \frac{1}{1-p^{-s}},
\end{align}
and the result follows.
\end{proof}
 
\subsection{The odd primes and 2}

The first natural refinement is to ask what is the parity of (i) the odd primes and (ii) the prime 2 in prime factorizations.  We have:

\begin{prop}
\label{2odd}
Assuming the Riemann hypothesis,
\be
\sum_{n\leq x}\lambda (n;2,1) = O_{\epsilon}(x^{1/2+\epsilon})
\ee
for all $\epsilon>0$, while unconditionally this sum is $o(x)$. On the other hand,
\be
\sum_{n\leq x}\lambda (n;2,2) = \frac{x}{3}+o(x)
\ee
and is nonnegative for all $x\geq 1$.
\end{prop}

\begin{proof}
We first treat the simpler case $L(n;2,2)$. The Dirichlet series
\be
D(s;2,2) = \zeta(s) \frac{1-2^{-s}}{1+2^{-s}}
\ee
has meromorphic continuation to $\Re(s)\geq1$ with only a simple pole at $s=1$ with residue $\frac13$, and is holomorphic for $\Re(s)>1$; hence we have
\be
\sum_{n\leq x}\lambda(n;2,2) = \frac{x}{3}+o(x)
\ee
by a standard Tauberian argument.
%the Ikehara--Wiener theorem.

Now, notice that $\lambda(n;2,2)$ is always $1,-1,1$ when $n$ is of the form $4k+1,4k+2,4k+3$ respectively. Only if it is of the form $4k$ can it take both $1$ and $-1$ as values, in which case it is determined by the value $\lambda(k;2,2)$. Thus the first few summands of $L(x;2,2)$ are
\be
1-1+1+\lambda(4;2,2)+1-1+1+\lambda(8;2,2)+\dots
\ee
and continuing thus, we conclude that for $L(x;2,2)\geq0$ for all $x\geq 1$.

Next, we treat the case $L(x;2,1)$. Notice that
\be
D(s;2,1) = \frac{\zeta(2s)}{\zeta(s)} \frac{1+2^{-s}}{1-2^{-s}},
\ee
where we recall from \eqref{liouvillezeta} that $\zeta(2s)/\zeta(s)$ is the Dirichlet series for the Liouville function. Since $\zeta(s)$ has a simple pole at $s=1$ and is nonvanishing on $\Re(s)=1$, it follows that $D(s;2,1)$ has analytic continuation to $\Re(s)\geq 1$ and has a zero at $s = 1$. Thus we have
\be
\sum_{n\leq x}\lambda(n;2,1) = o(x)
\ee
unconditionally. Assuming the Riemann hypothesis, $D(s;2,1)$ continues analytically to $\Re(s)>\frac12$ and has a simple pole at $s=\frac12$, and so $L(x;2,1)=O_{\epsilon}(x^{1/2+\epsilon})$.
\end{proof}

By a similar argument, we observe the analogous behaviour for general arithmetic progressions, described in Theorem \ref{nonres}.

\begin{proof}[Proof of Theorem \ref{nonres}]
Simply observe that the Dirichlet series in this setting can be expressed as
\be
D(s;q,a_1,\dots,a_{\varphi(q)})=\frac{\zeta(2s)}{\zeta(s)}\prod_{p \mid q}\frac{1+p^{-s}}{1-p^{-s}},
\ee
and
\be
D(s;q,b_1,\dots,b_k)=\zeta(s)\prod_{i=1}^k\prod_{p \mid b_i}\frac{1-p^{-s}}{1+p^{-s}},
\ee
and argue as in Proposition \ref{2odd}.
\end{proof}

%Qualitatively, we note that since $(p-1)/(p+1) <1$ for any prime $p$, the growth rate becomes slower and slower as more primes enter into the product.

The following corollary is proved in the same manner for the classical Liouville function, after the method of Landau.

\begin{cor}
With assumptions as in Theorem \ref{nonres},
\be
\sum_{n\leq x}\lambda (n;q,a_1,\dots,a_{\varphi(q)}) = o(x)
\ee
is equivalent to the prime number theorem.
\end{cor}
%I disagree with the referee; it is well-known what ``equivalent to the PNT'' means.

By a similar reasoning, we shall see that one can also recover Dirichlet's theorem on primes in arithmetic progressions, but only in cases where there are no complex Dirichlet characters modulo $q$.

\subsection{A Chebyshev-type bias}

We show a result using the properties developed above, which can be interpreted as: the number of prime factors of the form $4k+1$ and $4k+3$ \emph{both} tend to appear an even number of times, but the former having a much stronger bias. We first prove the following closed formula.

\begin{lem}
\label{lambda43}
Define the characteristic function
\be
c(x,k)=
\begin{cases}
1 & [x] \equiv 2k \pmod{4k}\\
0 & \text{otherwise}
\end{cases}
\ee
for any $x>0$. Then
\be
\sum_{n\leq x} \lambda(n;4,3) = \sum_{k=1}^\infty c(x,2^k),
\ee
with finitely many terms on the right-hand side being nonzero.
\end{lem}

\begin{proof}
To prove the formula, we repeatedly apply the elementary fact that if $n\equiv 1 \pmod{4}$ (respectively $3 \pmod{4}$), then the prime factors of $n$ of the form $3$ modulo $4$ appear an even (respectively odd) number of times. 

First, observe that $\lambda(n;4,3)$ is $1$ or $-1$ if $n$ is $1$ or $3$ modulo $4$, thus the sum
\be
\sum_{\substack{n\leq x \\ n \equiv 1 (2)}} \lambda(n;4,3)  
\ee
is equal to $c(n,2)$. Then we move on to the even numbers, which, written as $2m, 2(m+1)$ and using $\lambda(2;4,3)=1$, gives again the pattern $1$ and $-1$ depending on whether $m$ is $1$ or $3$ modulo $4$. The even numbers of the form $4$ and $6$ modulo $8$ contribute the term $c(n,2^2)$.

Repeating this process we obtain the terms $c(n,2^k)$ for all $k$, but certainly for $k$ large enough this procedure will cover all $n\leq x$, so only finitely many terms will be nonzero.
\end{proof}

\begin{prop}
\label{primesmod4}
Let $\chi_4$ be the nonprincipal Dirichlet character modulo $4$. Then
\be
\sum_{n\leq x}\lambda(n;4,1) = O_{\epsilon}(x^{1/2+\epsilon})
\ee
for any $\epsilon > 0$ assuming the generalised Riemann hypothesis for $L(s,\chi_4)$, while unconditionally this sum is $o(x)$. On the other hand,
\be
\sum_{n\leq x}\lambda(n;4,3) = O(\log x)
\ee
for any $\epsilon>0$ and is nonnegative for $x\geq 1$.
\end{prop}

\begin{proof}
We first prove the latter statement. The Dirichlet series of $\lambda(n;4,3)$ can be written as
\be
\label{Ds43}
D(s;4,3) = L(s,\chi_4) \frac{1}{1-2^{-s}}.
\ee
Thus we see that $D(s;4,3)$ has analytic continuation to the half-plane $\Re(s)>0$.

From the explicit formula in Lemma $\ref{lambda43}$ we also see immediately that $L(x;4,3)$ is nonnegative, and given any $C>0$ we can find $x$ large enough so that $L(x;4,3)>C$. Hence $D(s;4,3)$ has a simple pole at $s=0$, and finally we conclude that
\be
\sum_{n\leq x}\lambda(n;4,3) = O(\log x)
\ee
On the other hand, from \eqref{phiproduct} we have that
\be
D(s;4,1)D(s;4,3) = \zeta(2s) \frac{1+2^{-s}}{1-2^{-s}},
\ee
which by \eqref{Ds43} is
\be
D(s;4,1) = \frac{\zeta(2s)}{L(s,\chi_4)}(1+2^{-s})
\ee
Comparing both sides, we observe that $D(s;4,1)$ is analytic in the region $\Re(s)\geq1$, giving $o(x)$ unconditionally by analytic continuation of $L(s,\chi_4)$. Moreover, assuming the generalised Riemann hypothesis for $L(s,\chi_4)$, we see that $D(s;4,1)$ in fact converges absolutely in $\Re(s)>\frac12$, with only a simple pole at $s=\frac12$, so that
\be
\sum_{n\leq x}\lambda(n;4,1) = O_{\epsilon}(x^{1/2 + \epsilon}),
\ee
as required.
\end{proof}

The proposition above holds more generally for any $q\geq 2$, by the same method of proof, using the following observation: Let $r=\varphi(q)/2$. Then there is exactly one combination of residue classes, say $b_1,\dots, b_r$ such that
\be
D(s;q,b_1,\dots,b_r) = L(s,\chi_q) \prod_{p \mid q}\frac{1}{1-p^{-s}}
\ee
where $\chi_q$ is a nonprincipal real Dirichlet character modulo $q$; whereas
\be
D(s;q,a_1,\dots,a_{q-r}) = \frac{\zeta(2s)}{L(s,\chi_q)}\prod_{p \mid q}(1-p^{-s}),
\ee
where $a_1,\dots,a_{q-r}$ are the remaining residue classes.

\begin{proof}[Proof of Theorem \ref{varphi2}]
The estimate \eqref{resgeneral} follows the same argument as in Proposition \ref{primesmod4}. For \eqref{dirichletgeneral}, using the relation between $\chi_q$ and the Kronecker symbol, we need to extend the arguments of \cite[Corollary 6]{BCC} to composite moduli, which we do as follows:

Let $p$ be a prime dividing the modulus of $q$ the Kronecker symbol associated to $\chi_q$, and let $a_0+a_1p+\dots+a_kp^k$ be the base $p$ expansion of $n$. Then
\be
\sum_{i=1}^n \lambda(i;q,b_1,\dots,b_r) = \sum_{j=0}^k\sum_{i=1}^{a_j} \lambda(i;q,b_1,\dots,b_r)
\ee
as in \cite[Theorem 8]{BCC}. On the other hand, by the character-like property we know that $\lambda(i;q,b_1,\dots,b_r) = \lambda(kp+i;q,b_1,\dots,b_r)$ for $1\leq i \leq p-1$ and $k\in\N$, and hence $L(p^rn;q,b_1,\dots,b_r) =L(n;q,b_1,\dots,b_r)$ for any $r,n\in\N$. Using this, we may bound the maximum value by
\be
\max_{n<p^i}|L(n;q,b_1,\dots,b_r)| \ll \max_{n<q}|L(n;q,b_1,\dots,b_r)|
\ee
and thus conclude  $|L(x;q,b_1,\dots,b_r)|\ll \log x$. 
\end{proof}

\begin{rem}
We have not strived for the optimal unconditional bounds, that is, not assuming the generalised Riemann hypothesis. The relevant estimates can certainly be improved, for example, using the zero-free regions for the associated Dirichlet $L$-functions and $\zeta(s)$.
\end{rem}

\subsection{General arithmetic progressions}

Now we turn to general arithmetic progressions. We now restrict to residue classes $a$ coprime to $q$, which is the most interesting case. We will also assume moreover that $\varphi(q)>2$.

\begin{proof}[Proof of Theorem \ref{karatsuba}]
For the first statement, our method follows that of Karatsuba for the classical Liouville function \cite[Theorem 1]{K}, namely via the Selberg--Delange method. We begin with the following expression:
\begin{align}
D(s;q,a_1,\dots,a_r)
&=\zeta(s) \prod_{i=1}^r \prod_{p\equiv a_i \hspace{-.2cm} \pmod{q}} \frac{(1-p^{-s})^2}{1-p^{-2s}}\\
&=\zeta(s)^{1 - \frac{2r}{\varphi(q)}} \zeta(2s)^{\frac{r}{\varphi(q)}} \prod_{i=1}^r \exp(G_{a_i}(2s)-2G_{a_i}(s))\label{expGa}
\end{align}
by \eqref{faga}. Applying Perron's formula, we have
\be
\sum_{n\leq x}\lambda(n;q,a_1,\dots,a_r) = \frac{1}{2\pi i}\int_{b-iT}^{b+iT} D(s;q,a_1,\dots,a_r)\frac{x^s}{s} ds + O\left(\frac{x}{T}\right)
\ee
for $b>1$ and $x,T\geq 2$. Now analytically continue $D(s;q,a_1,\dots,a_r)$ to the left of the line $\Re(s)=1$, say to $\sigma\geq 1-C/\log T$, so that the integral is estimated by
\begin{align}
&\frac{1}{2\pi i}\int_{b-iT}^{b+iT} D(s;q,a_1,\dots,a_r)\frac{x^s}{s} ds \\
&= \frac{1}{2\pi i}\int_\gamma D(s;q,a_1,\dots,a_r)\frac{x^s}{s} ds + O\left(\frac{x}{T}\right),
\end{align}
where $\gamma$ is a closed loop around $s=1$, with radius taken to be less than $1-C/\log T$. In fact, we will choose $T$ such that $\log T = C(\log x)^{1/2}$. 

Now if we define the function $H(s)$ by the relation $\frac{1}{s}D(s;q,a_1,\dots,a_r) = (s-1)^{1 - \frac{2r}{\varphi(q)}} H(s),$ we can write
\be
\frac{1}{2\pi i}\int_\gamma D(s;q,a_1,\dots,a_r)\frac{x^s}{s} ds = xI(x)
\ee
where 
\be
I(x)=\frac{1}{2\pi i}\int_{\gamma'} H(s+1)x^s s^{1 - \frac{2r}{\varphi(q)}} ds
\ee
and $\gamma'$ is the contour obtained by translating $\gamma$ by $s\mapsto s+1$. Taking $\gamma$ to have radius $1/\sqrt{\log x}$, the integral can be written as (see \cite[p.216]{K})
\be
I(x)=\sum_{0\leq j\leq \sqrt{\log x}} \frac{B_j}{\Gamma\left(\frac{2r}{\varphi(q)} - j - 1\right) (\log x)^{2 + j - \frac{2r}{\varphi(q)}}} + O(e^{-C \sqrt{\log x}}),
\ee
where $B_j$ are the coefficients of the Taylor series expansion of $H(s)$ at $s=1$. In particular, $B_0 = H(1) >0$. Set $b_0 = B_0/\Gamma(\frac{2r}{\varphi(q)} - 1),$ and notice that $\Gamma(\frac{2r}{\varphi(q)} - 1)$ is positive or negative depending on whether $2r/\varphi(q)$ is greater or lesser than $1$, and is singular at $2r=\varphi(q)$. Since we may truncate the sum over $j$ to one term with an error term of size $O((\log x)^{\frac{2r}{\varphi(q)} - 3})$, the result follows.

The proof of the second case follows from an application of Ingham's analysis of $\lambda(n)$ in \cite{I}. Recall the Dirichlet series expression from \eqref{expGa}, we have that
\be
\sum_{n\leq x}\lambda(n;q,a_1,\dots,a_r) = \zeta(2s)^{1/2}\prod_{i=1}^r \exp(2G_{a_i}(s)-G_{a_i}(2s))
\ee
where we have used $r=\varphi(q)/2$.
\end{proof}

%%%%%%%%%%%%%%%%%%%%%%%%%%%%%%%%%%%%%%%%%%%%%%%%%%%%%%%%%%%
%%%%%%%%%%%%%%%%%%%%%%%%%SIGN CHANGE%%%%%%%%%%%%%%%%%%%%%%%%%%%%%%%%%%%%%%%%%%%%%%%%%%%%%%%%%%%%%%%%%%%%%%%%%%%%%%%%%%%%%%%%%

\section{Sign changes and biases: complementary case}
\label{biases}

\subsection{Sign Changes}

Let $r = \varphi(q)/2$, and let $a_1,\ldots,a_{q-r}$ denote the set of residue classes for which
\be
\label{D(s)}
D(s;q,a_1,\dots,a_{q-r}) = \frac{\zeta(2s)}{L(s,\chi_q)}\prod_{p \mid q}(1-p^{-s}),
\ee
so that $\lambda(n;q,a_1,\dots,a_{q-r})$ is the complement to a character-like function. It is natural to ask whether $\sum_{n\leq x}\lambda(n;q,a_1,\dots,a_{q-r})$ changes sign infinitely often; for $q = 1$, this is P\'{o}lya's conjecture. In this section, we give proofs that this is indeed the case conditional on several different hypotheses: firstly, that the generalised Riemann hypothesis for $L(s,\chi_q)$ is false; secondly, that it is true but there exist zeroes of $L(s,\chi_q)$ of order at least two; and thirdly, that it is true and the nonnegative imaginary ordinates of the zeroes are linearly independent over the rationals. We expect the third set of hypotheses to be the true properties.

\begin{rem}
For small values of $q$, one ought to be able to unconditionally prove an infinitude of sign changes via numerical calculations involving zeroes of $L(s,\chi_q)$, as in \cite{BT}. For arbitrary $q$, however, we do not know of any method that would be able to give an unconditional proof.
\end{rem}

We begin with the proof under the first hypothesis.

\begin{prop}[{Cf.~\cite[Theorem 2.6]{Hu}}]
Suppose that the generalised Riemann hypothesis for $L(s,\chi_q)$ is false, so that $\Theta := \sup\{\Re(\rho) : L(\rho,\chi_q) = 0\} > 1/2$. Then
\be
\liminf_{x \to \infty} \frac{\sum_{n\leq x}\lambda(n;q,a_1,\dots,a_{q-r})}{x^{\Theta - \epsilon}} < 0,\ee
and
\be
\limsup_{x \to \infty} \frac{\sum_{n\leq x}\lambda(n;q,a_1,\dots,a_r)}{x^{\Theta - \epsilon}} > 0
\ee
for every $\epsilon > 0$.
\end{prop}

\begin{proof}
The proof is via the same method as \cite[Theorems 15.2 and 15.3]{MV}. More precisely, we note that
\begin{align}
\label{F(s)}
F(s) &:= \int_{1}^{\infty} \left(x^{\Theta - \epsilon} \pm \sum_{n\leq x}\lambda(n;q,a_1,\dots,a_{q-r})\right) x^{-s} \, \frac{dx}{x} \\
&= \frac{C}{s - \Theta + \epsilon} + \frac{\zeta(2s)}{s L(s,\chi_q)}\prod_{p \mid q}(1-p^{-s})
\end{align}
for $\Re(s) > 1$, and the right-hand side extends to a meromorphic function on the right-half plane $\Re(s) > 1/2$ with no real poles but complex poles in the strip $\Theta - \epsilon \leq \Re(s) \leq \Theta$ at the zeroes of $L(s,\chi_q)$.

If $x^{\Theta - \epsilon} \pm \sum_{n\leq x}\lambda(n;q,a_1,\dots,a_{q-r})$ is always positive for sufficiently large $x$, on the other hand, then Landau's lemma \cite[Lemma 15.1]{MV} implies that both sides of \eqref{F(s)} extend to a holomorphic function on some right-half plane $\Re(s) > \sigma_0$, but has a singularity at the point $s = \sigma_0$. This yields the desired contradiction $L(\Theta,\chi_q) \neq 0$. The remaining case $L(\Theta,\chi_q) = 0$ is similar, using the construction detailed in \cite[Proof of Theorem 15.3]{MV}.
\end{proof}

Next, we prove an infinitude of sign changes under the second set of hypotheses.

\begin{prop}[{Cf.~\cite[Theorem 2.7]{Hu}}]
Assume the generalised Riemann hypothesis for $L(s,\chi_q)$ and that $L(s,\chi_q)$ has a zero of order $m \geq 2$. Then
\be
\liminf_{x \to \infty} \frac{\sum_{n\leq x}\lambda(n;q,a_1,\dots,a_{q-r})}{\sqrt{x} (\log x)^{m - 1}} < 0, \ee
and
\be
\limsup_{x \to \infty} \frac{\sum_{n\leq x}\lambda(n;q,a_1,\dots,a_{q-r})}{\sqrt{x} (\log x)^{m - 1}} > 0.
\ee
\end{prop}

\begin{proof}
This is via the same method as \cite[Theorem 15.3]{MV}; see in particular \cite[p.~467]{MV}.
\end{proof}
%
%PH: I think the referee is being pedantic here in asking for more details; this is a pretty standard calculation (for example, I omit the details in \cite{Hu}).

Of course, it is widely expected that $L(s,\chi_q)$ does indeed satisfy the generalised Riemann hypothesis and that all of its zeroes are simple. If we assume an additional widely believed conjecture on the behaviour of the zeroes of $L(s,\chi_q)$, namely the linear independence hypothesis, then we can again show that $\sum_{n\leq x}\lambda(n;q,a_1,\dots,a_{q-r})$ changes sign infinitely often.

\begin{prop}[{Cf.~\cite[Theorem 2.8]{Hu}}]
Assume the generalised Riemann hypothesis and the linear independence hypothesis for $L(s,\chi_q)$. Then
\be
\liminf_{x \to \infty} \frac{\sum_{n\leq x}\lambda(n;q,a_1,\dots,a_{q-r})}{\sqrt{x}} < 0
\ee
and
\be
\limsup_{x \to \infty} \frac{\sum_{n\leq x}\lambda(n;q,a_1,\dots,a_{q-r})}{\sqrt{x}} > 0.
\ee
\end{prop}

\begin{proof}
The proof is a straightforward modification of \cite[Proof of Theorem A]{I}, where the same result is proven for the Liouville function $\lambda(n)$,with the associated Dirichlet series being $\zeta(2s)/\zeta(s)$ in place of \eqref{D(s)}.
\end{proof}

\subsection{Biases}

Now we explain why there appears to be a positive bias in the limiting behaviour of $\sum_{n\leq x} \lambda(n;q,a_1,\dots,a_{q-r})$. This is for the same reason that $\sum_{n\leq x} \lambda(n)$ appears to have a heavy bias towards being negative: it is due to the fact that should it be the case that $L(1/2,\chi_q) > 0$, as is widely believed, then $D(s;q,a_1,\dots,a_{q-r})$ has a pole at $s = 1/2$ with positive residue. To quantify this more precisely, we work with the limiting distribution of $e^{-y/2} \sum_{n\leq e^y}\lambda(n;q,a_1,\dots,a_{q-r})$.

\begin{thm}
Assume the generalised Riemann hypotheses for $L(s,\chi_q)$, and that the bound
\be
\sum_{0 < \gamma \leq T} \frac{1}{|L'(\rho,\chi_q)|^2} \ll T^{\theta}
\ee
holds for some $1 \leq \theta < 3 - \sqrt{3}$. Then
\be
e^{-y/2} \sum_{n\leq e^y}\lambda(n;q,a_1,\dots,a_r)
\ee
has a limiting distribution $\mu$.

Suppose additionally that $L(s,\chi_q)$ satisfies the linear independence hypothesis. Then the Fourier transform $\widehat{\mu}$ of $\mu$ is given by
\be
\widehat{\mu}(\xi) = e^{-i c \xi} \prod_{\gamma > 0} J_0(|r(\gamma) \xi|),
\ee
where $J_0(z)$ is the Bessel function of the first kind,
\be
c := \frac{1}{L(1/2,\chi_q)} \prod_{p \mid q} (1 - p^{-1/2}),
\ee
and
\be
r(\gamma) := \frac{2 \zeta(2\rho)}{\rho L'(\rho,\chi_q)} \prod_{p \mid q} (1 - p^{-\rho}).
\ee
The mean and median of $\mu$ are both equal to $c$, while the variance of $\mu$ is equal to $\frac{1}{2} \sum_{\gamma > 0} |r(\gamma)|^2$.
\end{thm}

\begin{rem}
Once again, the restriction on $\theta$ may be weakened to $\theta < 2$.
\end{rem}

\begin{proof}
This essentially follows from \cite{ANS}. More precisely, \cite[Lemmata 4.3 and 4.4]{ANS} imply an explicit expression of the form
\be
e^{-y/2} \sum_{n\leq e^y} \lambda(n;q,a_1,\dots,a_{q-r}) = c + \Re\left(\sum_{0 < \gamma \leq X} r(\gamma) e^{i\gamma y}\right) + \mathcal{E}(y,X)
\ee
for $y > 0$ and $X \geq 1$, where
\be
\mathcal{E}(y,X) = O_{\epsilon}\left(\frac{y e^{y/2}}{X} + \frac{e^{y/2}}{yX^{1 - \epsilon}} + (X^{\theta - 2} \log X)^{1/2} + e^{-y(1/2 - b)}\right)
\ee
for $0 < \epsilon < b < 1/4$, $x > 1$; cf.~\cite[Proof of Corollary 1.6]{ANS}. With this in hand, the existence of $\mu$ follows from \cite[Theorem 1.4]{ANS}, while the identity for $\widehat{\mu}$ is a consequence of \cite[Theorem 1.9]{ANS}. Finally, the proof of the identities for the mean and variance follows the same lines as \cite[Proof of Corollary 6.3]{Hu}, while the proof of the identity for the median follows \cite[Proof of Theorem 5.1]{Hu}; cf.~\cite[Theorem 1.14]{ANS}.
\end{proof}

The proof of Theorem \ref{mainbiasesthm} may then be proven as a consequence of this.

\begin{proof}[Proof of Theorem \ref{mainbiasesthm}]
This follows, with minor modifications, via the same method as \cite[Proof of Theorem 1.5]{Hu}, where (among other things) the analogous result is proved for $\lambda(n)$ in place of $\lambda(n;q,a_1,\dots,a_{q-r})$.
\end{proof}

Finally, we mention that all of the results in this section are valid analogously for the case where $r = \varphi(q)$ and $a_1,\ldots,a_{\varphi(q)}$  is the set of residue classes coprime to $q$, so that
\be
D(s;q,a_1,\dots,a_{\varphi(q)}) = \frac{\zeta(2s)}{\zeta(s)}\prod_{p \mid q}\frac{1 + p^{-s}}{1-p^{-s}}.
\ee
This has a pole at $s = 1/2$ with negative residue; it is for this reason that there is a bias towards 
\[
\sum_{n\leq x}\lambda(n;q,a_1,\dots,a_{\varphi(q)})
\]
 being nonpositive.

Indeed, we can again conditionally prove the existence of a limiting distribution $\mu$ of 
\[
e^{-y/2} \sum_{n\leq e^y} \lambda(n;q,a_1,\dots,a_{\varphi(q)}), 
\]
but now
\be
c := \frac{1}{\zeta(1/2)} \prod_{p \mid q} \frac{1 + p^{-1/2}}{1-p^{-1/2}}
\ee
and
\be
r(\gamma) := \frac{2 \zeta(2\rho)}{\rho \zeta'(\rho)} \prod_{p \mid q} \frac{1 + p^{-\rho}}{1-p^{-\rho}}.
\ee
The mean and median of $\mu$ is $c$, which is negative; for this reason, the logarithmic density $\delta(P)$ of
\be
P := \left\{x \in [1,\infty) : \sum_{n\leq x}\lambda(n;q,a_1,\dots,a_{\varphi(q)}) \leq 0\right\}
\ee
is at least $1/2$ but strictly less than $1$, so that $\sum_{n\leq x}\lambda(n;q,a_1,\dots,a_{\varphi(q)})$ is nonpositive `most' of the time, yet it is positive a positive proportion of the time.

%%%%%%%%%%%%%%%%%%%%%%%%%%%%%%%%%%%%%%%%%%%%%%%%%%%%%%%%%%%
%%%%%%%%%%%%%%%%%%%%%%%%%%%%%%%%%%%%%%%%%%%%%%%%%%%%%%%%%%%
%%%%%%%%%%%%%%%%%%%%%%%%%%%%%%%%%%%%%%%%%%%%%%%%%%%%%%%%%%%

\subsubsection*{Note on computations}

All computations in this paper were done on Fortran95 and Python3. The largest computation that we could carry out was that of $L(x;4,1)$, which we verified to be positive for all $1< x \leq 10^{11}$. Most of the other calculations were carried out up to $10^{8}$ or $10^{9}$. Because of the limited amount of computer memory, we used an algorithm that calculates the parities in batches of size $10^8$. Also, to avoid the problem of factorizing large integers, we used multiplication to build up the parities of numbers up to $x$. This results in a significant increase in the speed without any type of parallel computation. The codes for the computations of various combinations were primarily written in Python, which allows to easily construct all the required combinations.

%%%%%%%%%%%%%%%%%%%%%%%%%%%%%%%%%%%%%%%%%%%%%%%%%%%%%%%%%%%%
% Bibliography
%%%%%%%%%%%%%%%%%%%%%%%%%%%%%%%%%%%%%%%%%%%%%%%%%%%%%%%%%%%%
\bibliographystyle{alpha}
\bibliography{Parity}

\end{document}